\documentclass{amsart}

\usepackage{graphicx} 
\usepackage{amsmath}
\usepackage{amsthm}
\usepackage{amssymb}
\usepackage{xfrac}
\usepackage{faktor}
\usepackage{mathtools}
\usepackage{mathrsfs}
\usepackage{tikz}
\usepackage{framed}
\usepackage{upgreek}
\usepackage{bbm}

\usepackage{enumitem}
\setlist[enumerate]{label=\emph{\roman*})}

\newtheorem{theorem}{Theorem}[section]
\newtheorem{lemma}[theorem]{Lemma}
\newtheorem{proposition}[theorem]{Proposition}

\newtheorem*{theorem*}{Theorem}
\theoremstyle{definition}
\newtheorem{definition}[theorem]{Definition}

\theoremstyle{remark}
\newtheorem*{remark}{Remark}

\def\epsilon{\varepsilon}
\def\phi{\varphi}
\def\<{\langle}
\def\>{\rangle}

\def\N{\mathbb{N}}

\def\R{\mathbb{R}}

\def\P{\mathcal{P}}
\def\D{\mathcal{D}}

\def\W{\mathrm{Wr}}
\def\M{\mathcal{M}}
\def\K{\mathcal{K}}
\def\Kn{\mathscr{K}}
\def\ev{\mathrm{ev}}
\def\uu{\mathcal{W}}

\title{Vassiliev invariants and writhe for periodic orbits of Axiom A flows}
\author{Solly Coles}\thanks{\textit{Department of Mathematics, Northwestern University}}\thanks{solly.coles@northwestern.edu}

\begin{document}
\begin{abstract}
    We obtain asymptotics for the average value taken by a Vassiliev invariant on knots appearing as periodic orbits of an Axiom A flow on $S^3.$ The methods used also give asymptotics for the writhe of periodic orbits. Our results are analogous to those of G. Contreras for average linking numbers.
\end{abstract}
\maketitle
\section{Introduction}
Axiom A flows are a class of chaotic dynamical system introduced by Smale in \cite{smale}. On a compact Riemannian manifold, the Axiom A flows form an open set in the $C^1$ topology. Furthermore, these systems are structurally stable, meaning that a small $C^1$ perturbation of an Axiom A flow yields a conjugate flow. Smale showed that for Axiom A flows, the non-wandering set is comprised of finitely many disjoint invariant attractors called basic sets. The dynamics of an Axiom A flow on one of its basic sets has been widely studied, and in particular the distribution of of periodic orbits is well understood. 

In \cite{contreras}, Contreras considers an Axiom A flow $X^t:S^3\to S^3$ restricted to a basic set $\Lambda$. It is shown that the average linking number of two periodic orbits whose periods are approximately $S$ and $T$ grows proportionally to $ST$. Precisely, let $$\P_T=\{\gamma\hbox{ a periodic orbit : period}(\gamma)\in (T-1,T]\},$$ and let $I:\Lambda\times \Lambda:\to \R$ be defined by $$I(x,y)=\frac{1}{4\pi}\frac{x-y}{\|x-y\|^3}\cdot (X(x)\times X(y)),$$ where $X$ is the vector field generating $X^t.$ Note that $I$ is undefined at the diagonal in $\Lambda\times \Lambda.$ One case of Theorem B in \cite{contreras} says that if $X^t$ is weak-mixing, then
\begin{equation}\label{contrerasthm}
\lim_{S,T\to\infty}\frac{\sum_{\gamma\in \P_S,\eta\in \P_T}\mathrm{lk}(\gamma,\eta)}{ST\#\P_S\#\P_T}=\int I\,d(\mu\times\mu),
\end{equation}
where $\mu$ is the measure of maximal entropy of $X^t$ on $\Lambda$, and $\mathrm{lk}$ is the linking number.
The scheme of proof is as follows. Integrating $I$ along $\gamma$ and $\eta$ gives their linking number, via the classical Gauss linking integral. Thus the left hand side of (\ref{contrerasthm}) becomes a limit of integrals of $I$ against a product of measures supported on the orbits in $\P_S$ and $\P_T$. A result of Bowen \cite{Bowen}, is that these periodic orbit measures converge weak* to $\mu$ (this is usually referred to as equidistribution). The remainder of the proof consists of showing that $I$ is sufficiently well-behaved near the diagonal for the integrals to converge. 

In this paper, we will prove analogous results where instead of the linking number, we consider Vassiliev invariants and writhe of periodic orbits. To state our main result precisely, we give a brief description of these quantities; the formal definitions will be given in Section \ref{vassilievsec}.

The Vassiliev (or finite type) knot invariants were introduced by Vassiliev in \cite{vassiliev}, and have since been shown to be a powerful class of invariants. The coefficients of all classical knot polynomials are Vassiliev invariants (up to coordinate changes). This is also true of the quantum invariants of Reshetikhin and Turaev (see \cite{Bar-Natan} or \cite{CDM}, for example). The writhe of a knot $K$, whilst not an invariant, measures the average amount of self crossing across all planar projections of $K$. This is in some sense a measure of the tangledness of $K$, and is used in the study of elastic rods, in particular DNA topology, to quantify deformation and coiling due to torsional stress (see \cite{fuller}, for example).

All Vassiliev invariants can be evaluated using the configuration space integrals of Bott and Taubes (see \cite{Bott-Taubes}, \cite{volic}). We give a brief outline here, with details in Section \ref{confspaceint}. 

Let $\Kn$ denote the space of knots (smooth embeddings $S^1\to \R^3$). For a manifold $M$, let $C(k,M)$ denote the $k$-fold configuration space of $M$. Then, for each Vassiliev invariant $V:\Kn\to \R$, there exist $m\in \N$ and functions $f_k\colon C(k,S^1)\times \Kn\to \R$ for $k=2,\ldots, m$   such that 
\begin{equation}\label{vassexpression}
V(K)=\sum_{k=2}^m\int_{z\in C(k,S^1)}f_k(z,K)\,dz.
\end{equation}
Each function $f_k$ corresponds to a collection of trivalent diagrams, which are trivalent graphs consisting of $k\geq 2$ vertices lying on a circle, and $s\geq 0$ vertices inside the circle. Given such a diagram $D$, one considers the Gauss maps 
\begin{align*}
    h_{ij}\colon C(k+s,\R^3)&\to S^2\\
    (x_1,\ldots, x_{k+s})&\mapsto \frac{x_j-x_i}{\|x_j-x_i\|},
\end{align*}
where $i<j$ label a pair of adjacent vertices in $D$. The core of the construction involves taking a product of pullbacks of the standard volume form $\omega$ on $S^2$, under the maps $h_{ij}$. One must then perform an appropriate fibre integration and evaluate the resulting form on tangent vectors to $K\in \Kn$. This process gives a function $f_D\colon C(k,S^1)\times \Kn\to \R$, and $f_k$ in (\ref{vassexpression}) is a weighted sum $\sum_{k(D)=k}W(D)f_D$. Strictly speaking, to ensure pullbacks are nontrivial and extend smoothly to the boundary, one should instead carry out this procedure in the Fulton-MacPherson compactification of configuration space.

For the simplest trivalent diagram $D_0$, with $k=2$ and $s=0,$ the function $f_{D_0}$ is given by $$f_{D_0}(t_1,t_2,K)=\frac{K(t_2)-K(t_1)}{\|K(t_2)-K(t_1)\|^3}\cdot (K'(t_1)\times K'(t_2)),$$ which resembles the function in the Gauss linking integral and that considered by Contreras. Furthermore, the integral of $\frac{1}{4\pi}f_{D_0}(\cdot,\cdot,K)$ defines the writhe of $K$. In \cite{KV}, Komendarczyk and Voli\'c analyse the functions $f_D$ along knots given by closing up flow trajectories using short geodesics. They show that the behaviour of $f_D$ near the boundary of configuration space is similar to that of $f_{D_0}.$ The purpose of this paper is bring together their analysis with the appropriate equidistribution theory to obtain the following.

\begin{theorem}\label{mainthmintro}
    Let $X^t$ be a weak-mixing Axiom A flow on $S^3$, restricted to a basic set $\Lambda$. Let $V:\Kn\to \R$ be a Vassiliev invariant and take $m,f_k$ as in $(\ref{vassexpression})$. Then  
    $$\lim_{T\to \infty}\frac{\sum_{\gamma\in\P_T}V(\gamma)}{T^{m}\#\P_T}=\int\sum_{k(D)=m}W(D) f_D\,d(\underbrace{\mu\times \cdots \times \mu}_{m \hbox{ times}}),\hbox{ and }$$
  $$\lim_{T\to \infty}\frac{\sum_{\gamma\in \P_T}\W(\gamma)}{T^2\#\P_T}=\frac{1}{4\pi}\int f_{D_0}\,d(\mu\times\mu),$$
where $\mu$ is the measure of maximal entropy for $X^t$ on $\Lambda.$
\end{theorem}

The outline of the paper is as follows. In Section 2, we define Axiom A flows on $S^3$. In Section 3, we describe the configuration space integrals of Bott and Taubes, and in Section 4 we show how Vassiliev invariants and writhe can be evaluated using these integrals. In Section 5 we find asymptotics for the average value taken by a configuration space integral over the periodic orbits of our flow, which proves Theorem \ref{mainthmintro}.

\subsection*{Acknowledgements} I am grateful to  Bryna Kra and Richard Sharp for many useful discussions and their comments about this paper. I also thank Rafal Komendarczyk for helpful insight on the paper \cite{KV}.

\section{Axiom A flows}
Let $X$ be a non-stationary $C^1$ vector field on $S^3$ generating a flow $X^t\colon S^3\to S^3.$ Given $\Lambda\subset S^3$, let $T_\Lambda S^3$ denote the restriction of the tangent bundle $TS^3$ to points in $\Lambda$, i.e. $T_\Lambda$ is the disjoint union, over $x\in \Lambda$, of the tangent spaces $T_xS^3$. 

A closed invariant set $\Lambda$ is \textit{hyperbolic} if there is a continuous $DX^t$-invariant splitting $T_\Lambda S^3=E^u\oplus E\oplus E^s$ and constants $C,\lambda>0$ satisfying:
\begin{enumerate}
    \item $\|DX^t(v)\|\leq Ce^{-\lambda t}\|v\|$ for all $v\in E^u$, and $t\leq 0.$
    \item $\|DX^t(v)\|\leq Ce^{-\lambda t}\|v\|$ for all $v\in E^s$, and $t\geq 0.$
    \item $E$ is the subbundle generated by $X.$
\end{enumerate}
A hyperbolic set $\Lambda$ is called \textit{basic} if:
\begin{enumerate}
    \item $X|_\Lambda$ has an orbit which is dense in $\Lambda$.
    \item Periodic points of $X|_\Lambda$ are dense in $\Lambda$.
    \item There is an open set $U\subset M$ with $\Lambda=\bigcap_{t\in \R}X^t(U).$
\end{enumerate}
The flow $X^t$ is \textit{Axiom A} if its non-wandering set $\Omega\subset S^3$ is a disjoint union of  finitely many basic sets. In this paper, we will always consider $X^t$ to be restricted to one of these basic sets $\Lambda$. 

A classification of non-singular Axiom A flows on $S^3$ was given by Franks in \cite{franks} (this was extended to the singular case by de Rezende \cite{derezende}). We will not discuss this classification, but we note that these results imply there are many such flows: in particular, the suspension of any mixing subshift of finite type can be realised as an Axiom A flow on a basic set in $S^3.$

Note that we cannot have $\Lambda=S^3$, as then $X^t$ would be an Anosov flow, and $S^3$ carries no such flows (J. Plante and W. Thurston \cite{plante}). Thus $\Lambda\subsetneq S^3$ can be viewed as a compact subset of $\R^3.$ We will use this viewpoint when considering knot invariants for periodic orbits.

We will also assume that our flow is weak-mixing, meaning that if $f\colon \Lambda\to S^1$ is continuous, and there is $a\in [0,2\pi)$ such that $f\circ X^t(x)=e^{iat}f(x)$ for all $x$, then $a=0$ and $f$ is constant.

\section{Configuration space integrals}
We put aside the dynamics for now, and discuss Bott-Taubes integration, which will be used to evaluate Vassiliev invariants and writhe. A detailed survey of this method is given in \cite{volic}, though our notation more closely resembles the exposition in \cite{KV}. We will define a class of integrals, each corresponding to a \textit{trivalent diagram}.
\subsection{Trivalent diagrams}
A trivalent diagram $D$ is a connected graph consisting of $k=k(D)$ vertices lying on the same circle, and $s=s(D)$ inside of this circle (called free vertices). These are connected by edges in such a way that free vertices have valence 3, and circle vertices have valence 1. Note that the circle's edge is not considered as part of the edge set $E(D)$, so we have $\#E(D)=\frac{k+3s}{2}.$ If the circle edge were included, circle vertices would also have valence 3, hence the name trivalent diagram. The vertex set $V(D)$ has cardinality $k+s$, which must be even. Say $k+s=2n$ and label vertices with the set $\{1,\ldots, 2n\}$. Edges will be denoted by $(i,j),$ where $i,j\in \{1,\ldots,2n\}$ and $i<j$. The integer $n$ is called the degree of $D.$ Examples of these diagrams can be seen in Figure \ref{fig:3}.
\begin{figure}[ht]
\centering
\begin{tikzpicture}[scale=0.4]
\draw[color=blue] (0,0) circle(3);
    \draw[fill=black] (-3,0) circle(3pt);
    \draw[fill=black] (3,0) circle(3pt);
    \draw (-3,0)--(3,0);
    \draw (0,-4) node {$k=2,$ $s=0$};
\draw[color=blue] (8,0) circle(3);
    \draw[fill=black] (10.225,2) circle(3pt);
    \draw[fill=black] (5.775,2) circle(3pt);
    \draw[fill=black] (8,-3) circle(3pt);
    \draw[fill=black] (8,0) circle(3pt);
    \draw (8,-3) -- (8,0) -- (10.225,2);
    \draw (8,0) -- (5.775,2);
    \draw (8,-4) node {$k=3$, $s=1$};
\draw[color=blue] (16,0) circle(3);
    \draw[fill=black] (18.225,2) circle(3pt);
    \draw[fill=black] (13.775,2) circle(3pt);
    \draw[fill=black] (18.225,-2) circle(3pt);
    \draw[fill=black] (13.775,-2) circle(3pt);
    \draw[fill=black] (17,0) circle(3pt);
    \draw[fill=black] (15,0) circle(3pt);
    \draw (18.225,2) -- (17,0) -- (15,0) -- (13.775,2)
    (15,0) -- (13.775,-2)
    (17,0) -- (18.225,-2);
    \draw (16,-4) node {$k=4$, $s=2$};
\draw[color=blue] (24,0) circle(3);
    \draw[fill=black] (26.225,2) circle(3pt);
    \draw[fill=black] (21.775,2) circle(3pt);
    \draw[fill=black] (24,-3) circle(3pt);
    \draw[fill=black] (24,0) circle(3pt);
    \draw[fill=black] (26.225,-2) circle(3pt);
    \draw[fill=black] (21.775,-2) circle(3pt);
    \draw (24,-3) -- (24,0) -- (26.225,2);
    \draw (24,0) -- (21.775,2);
    \draw (26.225,-2) -- (21.775,-2);
    \draw (24,-4) node {$k=5$, $s=1$};
\end{tikzpicture}
\caption{Trivalent diagrams}
\label{fig:3}
\end{figure}
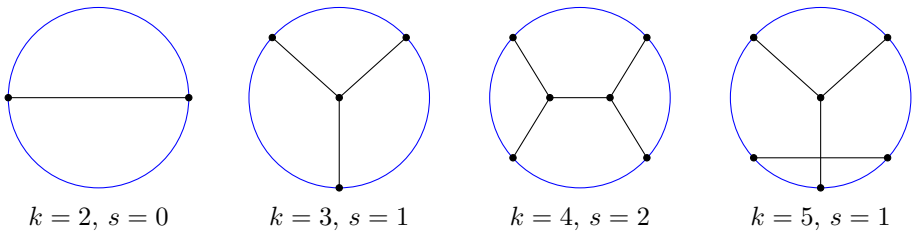

We denote by $TD_n$ the set of trivalent diagrams of degree $n$ up to orientation preserving diffeomorphism of the outer circle. Let $\D_n$ be the vector space over $\R$ generated by $TD_n$, modulo the relation (called the STU relation) in Figure \ref{fig:4}.
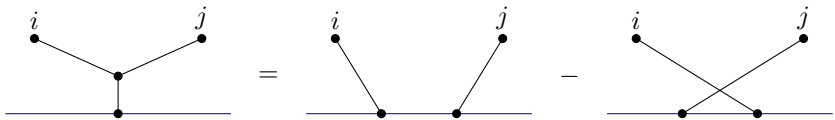
\begin{figure}[ht]
\centering
\begin{tikzpicture}[scale=0.5]
\draw[color=blue] (-3,0)--(3,0);
    \draw[fill=black] (-2.225,2) circle(3pt);
    \draw (-2.225,2) node[above] {$i$};
    \draw (2.225,2) node[above] {$j$};
    \draw[fill=black] (0,1) circle(3pt);
    \draw[fill=black] (0,0) circle(3pt);
    \draw[fill=black] (2.225,2) circle(3pt);
    \draw (-2.225,2) -- (0,1) -- (0,0)
    (0,1) -- (2.225,2);
\draw (4,1) node {$=$};
\draw[color=blue] (11,0) -- (5,0);
\draw (12,1) node {$-$};
    \draw[fill=black] (9,0) circle(3pt);
    \draw[fill=black] (7,0) circle(3pt);
    \draw[fill=black] (10.225,2) circle(3pt);
    \draw[fill=black] (5.775,2) circle(3pt);
    \draw (10.225,2) node[above] {$j$};
    \draw (5.775,2)  node[above] {$i$};
    \draw (10.225,2) -- (9,0) (7,0)-- (5.775,2);
\draw[color=blue] (13,0) -- (19,0);
    \draw[fill=black] (17,0) circle(3pt);
    \draw[fill=black] (15,0) circle(3pt);
    \draw[fill=black] (18.225,2) circle(3pt);
    \draw[fill=black] (13.775,2) circle(3pt);
    \draw (13.775,2) node[above] {$i$};
    \draw (18.225,2) node[above] {$j$};
    \draw (13.775,2) -- (17,0) (15,0) -- (18.225,2);
\end{tikzpicture}
\caption{STU relation: baseline is a segment of the circle.}
\label{fig:4}
\end{figure}
Consider the dual $\mathcal{W}_n$ of $\D_n$. We call $W\in \uu_n$ a \textit{weight system}, which is \textit{primitive} if it vanishes on diagrams which can be obtained as the connected sum of two smaller diagrams.

The integral corresponding to a diagram $D$ will be defined over a compactification of configuration space, defined by Fulton and MacPherson \cite{fultonmacpherson}.

\subsection{Fulton-MacPherson compactification}

Let $M$ be a Riemannian manifold and $C(k,M)$ denote the $k$-fold configuration space $$C(k,M)=\{(x_1,\ldots, x_k)\in M^k\hbox{ : }x_i\neq x_j\hbox{ whenever }i\neq j\}.$$
For each $S\subset \{1,\ldots, k\}$ with $|S|\geq 2,$ let $\Delta_S$ denote the $S$-diagonal in $M^k$ i.e. $$\Delta_S=\{(x_1,\ldots, x_k)\in M^k\hbox{ : }x_i=x_j\hbox{ for all }i,j\in S\}.$$
Let $\mathrm{Bl}(M^k,\Delta_S)$ be the blowup of $M^k$ along $\Delta_S$ i.e. the replacement of $\Delta_S$ with its unit normal bundle. Blowups in this context are discussed in more detail in \cite{KV}. For each $S$ there is a natural inclusion map $\beta\colon M^k\setminus \Delta_S\hookrightarrow \mathrm{Bl}(M^k,\Delta_S)$. Thus we can define a further inclusion $$\alpha^k_M\colon C(k,M)\hookrightarrow M^k\times \prod_{|S|\geq 2}\mathrm{Bl}(M^k,\Delta_S),$$ as the product of the standard inclusion and each of the blowup maps. The Fulton-MacPherson compactification of $C(k,M)$ is then defined by $$C[k,M]=\overline{\alpha^k_M(C(k,M))}.$$

We now define a pullback bundle over which we will later integrate. This requires two maps between appropriate compactifications. First, consider the evaluation map 
\begin{align*}
\mathrm{ev}\colon C(k,S^1)\times \Kn &\to C(k,\R^3) \\ 
(t_1,\ldots, t_k,K)&\mapsto (K(t_1),\ldots, K(t_k)),
\end{align*}
and let $\widetilde{\ev}$ be the lift of $\ev$ to $C[k,S^1]\times \Kn.$ Next, fix $s\geq 0$ and let $$\pi_k\colon C(k+s,\R^3)\to C(k,\R^3)$$ be the projection to the first $k$ entries. Denote by $\widetilde{\pi}_k$ the lift of $\pi_k$ to $C[k+s,\R^3].$ Define $C[k,s,\R^3,\Kn]$ as the pullback bundle given by $\widetilde{\ev}$ and $\widetilde{\pi}_k,$ as in Figure \ref{fig:5}. 
\begin{figure}[ht]
\centering
\begin{tikzpicture}[scale=0.5]
\draw (0,0) node[left] {$C[k,S^1]\times \Kn$};
\draw[->] (0,0) -- (5,0) node[midway, above] {$\widetilde{ev}$};
\draw (5,0) node[right] {$C[k,\R^3]$};
\draw[->] (-2.25,3) -- (-2.25,0.6);
\draw (-2.25,3) node[above]{$C[k,s,\R^3,\Kn]$};
\draw[->] (6.5,3) -- (6.5,0.6) node[midway, right]{$\widetilde{\pi}_k$};
\draw (6.5,3.5) node {$C[k+s,\R^3]$};
\draw[->] (0,3.5) -- (4.5,3.5);
\end{tikzpicture}
\caption{Defining $C[k,s,\R^3,\Kn]$.}
\label{fig:5}
\end{figure}

Letting  $p\colon C[k,S^1]\times \Kn\to \Kn$ be the projection onto the second component,  $C[k,s,\R^3,\Kn]$ can be understood as follows. 

For a fixed $K\in \Kn$, the interior of the fibre $(\widetilde{\pi}_k\circ p)^{-1}(K)$ in $C[k,s,\R^3,\Kn]$ is given by the points $$\{(x_1,\ldots, x_{k+s})\in C(k+s,\R^3)\hbox{ : }x_i\hbox{ lies on }K\hbox{ for }1\leq i\leq  k\}.$$
Such points are then extended to the boundary via the map $\alpha^{k+s}_{\R^3}$ defined above. 

\subsection{Configuration space integrals}\label{confspaceint}

Here we define the configuration space integral associated to a trivalent diagram $D.$ 

Let $k=k(D)$ and $s=s(D).$ For $(i,j)\in E(D)$, let $h_{ij}\colon C(k+s,\R^3)\to S^2$ denote the Gauss map $$h_{ij}(x_1,\ldots,x_{k+s})=\frac{x_j-x_i}{\|x_j-x_i\|},$$
and let $\widetilde{h}_{ij}$ be the lift of $h_{ij}$ to $C[k,s,\R^3,\Kn].$ Then, define $$h_D=\Bigg(\prod_{(i,j)\in E(D)}\widetilde{h}_{ij}\Bigg)\colon C[k,s,\R^3,\Kn]\to \prod_{(i,j)\in E(D)}S^2.$$ Pulling back the standard volume form $\omega$ on $S^2$, we obtain $$\omega_D=h_D^*(\omega\times\cdots \times \omega).$$
Since $\#E(D)=\frac{k+3s}{2}$, $\omega_D$ is a $(k+3s)$-form on $C[k,s,\R^3,\Kn].$ 

\begin{definition}
    Given $K\in \Kn$, define $I_D(K)$ to be the integral of $\omega_D$ over the $(k+3s)$-dimensional fibre of $K,$ i.e. $$I_D(K)=\int_{(\widetilde{\pi}_k\circ p)^{-1}(K)}\omega_D.$$
\end{definition} Let us give a more practical expression for $I_D(K).$ Note that the function $h_D$ above factors through a map $$\overline{h}_D\colon C[k+s,\R^3]\to \prod_{(i,j)\in E(D)}S^2,$$ 
so we may consider the $(k+3s)$-form $\beta=(\overline{h}_D)^*(\omega\times\cdots\times \omega)$ on $C[k+s,\R^3]$. Taking the fibre integral of $\beta$ with respect to $\widetilde{\pi}_k$ gives a $k$-form on $C[k,\R^3]$ (since the fibres $\widetilde{\pi}_k^{-1}(x)$ are $3s$-dimensional), we call this $k$-form $\varpi_D$. Let $\alpha=\alpha_{\R^3}^k$ be the map in the definition of the compactification, so that $\alpha^*\varpi_D$ is a $k$-form on $C(k,\R^3).$ 

For each $K\in \Kn,$ define a function 
\begin{align*}
    f_{D,K}\colon C(k,S^1)&\to \R\\
    (t_1,\ldots, t_k)&\mapsto (\alpha^*\varpi_D)_{(K(t_1),\ldots, K(t_k))}(K'(t_1),\ldots, K'(t_k)),
\end{align*}
where we have abused notation and used $K'(t_i)$ to denote the $3k$-vector with $K'(t_i)$ in entries $3i-3, 3i-2,3i-1$, and zeroes everywhere else. Proposition 3.7 in \cite{KV} says that $$I_D(K)=\int_{C(k,S^1)}f_{D,K}\,dt_1\ldots dt_k.$$

\section{Knot quantities as integrals}
Here we will describe the Vassiliev invariants and writhe, and show how they can be computed using the configuration space integrals above. In doing so, we will show that to prove Theorem \ref{mainthmintro} it suffices to find asymptotics for each $I_D$ on periodic orbits.
\subsection{Writhe}\label{writhe}
In this paper, a knot is a smooth embedding $K:S^1\to \R^3.$ We will abuse notation and also use $K$ to denote the image of such an embedding, paired with an orientation given by the clockwise orientation of $S^1.$

We begin by defining the writhe of $K\in\Kn$ as an average self-crossing number seen amongst the planar projections of $K$. Let $$G_K=\{v\in S^2\hbox{ : }v\hbox{ is not parallel to any tangent vector of }K\}.$$ Note that $G_K$ is open and dense, and has full volume in $S^2.$  For $v\in G_K$, consider the \textit{knot diagram} of $K$ on the plane normal to $v.$ This is a drawing of $K$ given by projecting $K$ onto the plane normal to $v,$ and replacing self intersections with \textit{crossings}, according to the order of points of $K$ along the directed line $\{tv\}_{t\in \R}$ in $\R^3$ (see Figure \ref{fig:6}).
\begin{figure}[ht]
\includegraphics[scale=0.35]{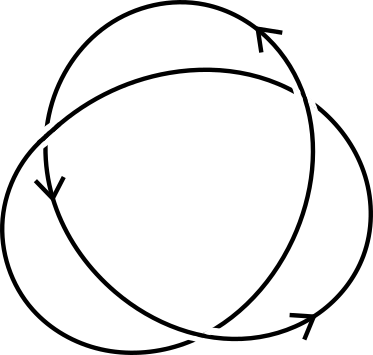}
\caption{A diagram for the trefoil knot}
\label{fig:6}
\end{figure}

The \textit{directional writhing number} $d_K(v)$ of a diagram is given by assigning each crossing either $+1$ or $-1$ depending on orientation (see Figure \ref{fig:1}), and taking the sum over all crossings. If $v\notin G_K$, we say $d_K(v)=0.$ For example, the diagram in Figure \ref{fig:6} has directional writhing number 3.

\begin{figure}[ht]
\centering
\begin{tikzpicture}[scale=0.5]
\draw[-] (-3,-2)--(-3,-0.25);
\draw[->] (-3,0.25)--(-3,2);
\draw[->] (-5,0)--(-1,0);
\draw[->] (3,-2)--(3,2);
\draw[-] (1,0)--(2.75,0);
\draw[->] (3.25,0)--(5,0);
\path (-3,-3) node {$+1$}
(3,-3) node {$-1$};
\end{tikzpicture}
\caption{Computing the directional writhing number}
\label{fig:1}
\end{figure}
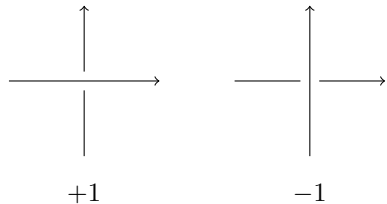
One can show that $d_K(v)$ is locally constant on $G_K$, and integrable over $S^2.$ 
\begin{definition}
The writhe of $K,$ written $\W(K)$, is defined by $$\W(K)=\int_{S^2}d_K(v)\, dv.$$
\end{definition}
Though this definition gives an intuitive notion of writhe, we will use another definition which gives a more direct formula, via the configuration space integral $I_{D_0}.$ 

\begin{theorem}\label{betterdef} If $K$ is differentiable, then \begin{align*}
4\pi\W(K)&=\int_{S^1}\int_{S^1}\frac{K(t_2)-K(t_1)}{\|K(t_2)-K(t_1)\|^3}\cdot (K'(t_2)\times K'(t_1))\,dt_1dt_2,\\
&=\int_{C(2,S^1)}f_{D_0,K}\,dt_1 dt_2=I_{D_0}(K).
\end{align*}
\end{theorem}
\begin{remark}
    A proof of this theorem can be found in \cite{AKT}, for example. They in fact show that the configuration space integral coincides with the average linking number $$\int_{v\in G_K}\mathrm{lk}(K,K+\epsilon v)\, dv,$$ where $\epsilon>0$ may depend on $v.$ One can see that $\mathrm{lk}(K,K+\epsilon v)=d_v(K)$ using the knot diagram computation of linking number.
\end{remark}

\subsection{Vassiliev invariants}\label{vassilievsec}
Here we will define Vassiliev (or finite type) invariants and show how they are evaluated with configuration space integrals.

An isotopy invariant $V:\Kn\to \R$ can be extended to knots with finitely many singularities (self intersections) by evaluating as in Figure \ref{fig:2} at each singularity, one-by-one. Extending $V$ to a knot with $n$ singularities involves evaluating $V$ at $2^n$ knots.
\begin{figure}[ht]
\centering
\begin{tikzpicture}[scale=0.5]
\draw[->] (-9.5,-2)--(-9.5,2);
\draw[->] (-11.5,0)--(-7.5,0);
\draw[-] (-3,-2)--(-3,-0.25);
\draw[->] (-3,0.25)--(-3,2);
\draw[->] (-5,0)--(-1,0);
\draw[->] (3.625,-2)--(3.625,2);
\draw[-] (1.625,0)--(3.375,0);
\draw[->] (3.875,0)--(5.625,0);
\draw[fill=red, color=red] (-9.5,0) circle(4pt);
\path (-6.25,0) node {$\Bigg)=V\Bigg($}
(0.25,0) node {$\Bigg)-V\Bigg($}
(6,0) node {$\Bigg)$}
(-12.125,0) node {$V\Bigg($};
\end{tikzpicture}
\caption{Extending $V$ to singularities}
\label{fig:2}
\end{figure}
\begin{definition}
    The invariant $V$ is a Vassiliev invariant of type $n$ if it vanishes on all knots with $n+1$ singularities.
\end{definition}

The following theorem, due to Altschuler and Freidel \cite{AF} (later reproved in the form below by D. Thurston \cite{Thurston}), says that all Vassiliev invariants can be expressed as a linear combination of configuration space integrals.

\begin{theorem}\label{vassinv}
Given a primitive weight system $W\in \mathcal{W}_n$, there are real numbers $\{m_D\}_{D\in TD_n}$ such that the map $$V_W(K)=\sum_{D\in TD_n}W(D)I_D(K)-m_D\W(K)$$ is a Vassiliev invariant of order $n.$ Furthermore, every Vassiliev invariant of order $n$ can be obtained in this way.
\end{theorem}

\begin{remark}
    The term $m_D\W(K)$ in the above theorem provides a correction term, known as the anomalous correction, which accounts for the integral of $\omega_D$ on the boundary of $(\widetilde{\pi}_k\circ p)^{-1}(K).$ 
\end{remark}

\section{Average integral values}

In this section, we will prove the main result of the paper. Recall that we are considering periodic orbits of a weak-mixing Axiom A flow $X^t$ on a basic set $\Lambda\subsetneq S^3.$ Given a trivalent diagram $D$, we define a function similar to $f_{D,K}$ which accounts for all periodic orbits of $X^t$ at once. With the notation from Section \ref{confspaceint}, let $k=k(D)$ and define 
\begin{align*}
 f_{D,X}\colon C(k,S^3)&\to \R\\
 (x_1,\ldots,x_k)&\mapsto (\alpha^*\varpi_D)_{(x_1,\ldots,x_k)}(X(x_1),\ldots X(x_k)).
\end{align*}
Let us now state the main result more precisely than in the introduction. To ease notation, for a measure $\nu$ on $\Lambda$, and $m\in \N$, let $\nu^m=\underbrace{\nu\times \cdots \times \nu}_{m \text{ times}}$ on $\Lambda^m.$
\begin{theorem}\label{mainthm1}
    Let $V_W:\Kn\to \R$ be a Vassiliev invariant of order $n$, and $\mu$ be the measure of maximal entropy for $X^t$ on $\Lambda.$ 
Then both of the following limits exist and the equalities hold.
    $$\lim_{T\to \infty}\frac{\sum_{\gamma\in \P_T}\W(\gamma)}{T^2\#\P_T}=\frac{1}{4\pi}\int f_{D_0,X}\,d\mu^2,$$
    
    $$\lim_{T\to \infty}\frac{\sum_{\gamma\in\P_T}V_W(\gamma)}{T^{k}\#\P_T}=\sum_{\substack{D\in TD_n \\ k(D)=k}}W(D)\int f_{D,X}\,d\mu^k,$$
where $k=\max\{k(D)\hbox{ : }D\in TD_n\hbox{ and }W(D)\neq 0\}.$

\end{theorem}

By Theorems \ref{betterdef} and \ref{vassinv}, we can express $V_W$ and $\W$ in terms of the integrals $I_D$, so defining $$A_D(T):=\frac{\sum_{\gamma\in \P_T}I_D(\gamma)}{\#\P_T},$$ it suffices to prove the following.

\begin{theorem}\label{aveintvalsthm} 
For $n\in \N$ and $D\in TD_n$, $$\lim_{T\to \infty}\frac{A_D(T)}{T^k}=\int f_{D,X} \,d\mu^k,$$ where $k=k(D).$  
\end{theorem}

To prove Theorem \ref{aveintvalsthm}, we essentially follow the method of Contreras \cite{contreras}, using control on $f_{D,X}$ obtained in \cite{KV} by Komendarczyk-Voli\'c.

For each $\gamma\in \P_T$, let $\nu_\gamma$ be the Borel measure given by $$\int \psi \,d\nu_\gamma=\int_0^{\ell(\gamma)}\psi(X^t(x_\gamma))\,dt,$$ where $\ell(\gamma)$ is the minimal period of $\gamma$, and $x_\gamma$ is any point on $\gamma.$ By definition of $f_{D,X}$, we have $$I_D(\gamma)=\int f_{D,X}\,d\nu_\gamma^k,$$ and therefore that $$\frac{A_D(T)}{T^k}=\int f_{D,X} \,d\nu_{T,k},$$ where $$\nu_{T,k}=\frac{\sum_{\gamma\in \P_T}\nu_\gamma^k}{T^k\#\P_T}.$$ For convenience, we will work with a different family of probability measures asymptotic to the $\nu_{T,k}$. If we set $\mu_\gamma=\frac{\nu_\gamma}{\ell(\gamma)},$ and $$\mu_{T,k}=\frac{\sum_{\gamma\in \P_T}\mu_\gamma^k}{\#\P_T},$$ the following clearly holds.

\begin{lemma}
    If either of the following limits exist, they both exist and the equality holds:
    $$\lim_{T\to\infty}\frac{A_D(T)}{T^k}=\lim_{T\to \infty}\int f_{D,X}\,d\mu_{T,k}.$$
\end{lemma}
\begin{proof}
    This follows from the argument above and the fact that $$\frac{(T-1)^k}{T^k}\mu_{T,k}\leq \frac{\nu_{T,k}}{T^k}\leq \mu_{T,k}.$$
\end{proof}

To complete the proof of Theorem \ref{mainthm1} we will show that for the measure of maximal entropy $\mu,$ $\int f_{D,X}\, d\mu^k$ exists and is equal to $$\lim_{T\to \infty}\int f_{D,X}\,d\mu_{T,k}.$$

We first show that $\mu_{T,k}$ converges weak* to the product $\mu^k$. For this, we use the large deviation theory developed in this context by Kifer. Denote by $\M(X)$ the space of invariant Borel probability measures for $X,$ with the weak* topology.

\begin{theorem}[Kifer, \cite{kifer}]\label{largedev}
    Given a compact subset $\K\subset \M(X)$ with $\mu\notin \K,$ we have $$\limsup_{T\to \infty}\frac{1}{T}\log\frac{\#\{\gamma\in \P_T\hbox{ : }\mu_\gamma\in \K\}}{\#\P_T}<0.$$
\end{theorem}
An appropriate choice of $\K$ in Theorem \ref{largedev} now gives $\mu_{T,k}\to \mu^k.$
\begin{theorem}\label{measconv}
    The measures $\mu_{T,k}$ converge weak* to the product $\mu^k.$
\end{theorem}
\begin{proof}
    Let $\psi\in C(M^k,\R),$ and fix $\epsilon>0$. Let $\K$ be the compact set $$\K=\bigg\{m\in\M(X)\hbox{ : }\bigg|\int\psi\,dm^k-\int\psi\,d\mu^k\bigg|\geq \epsilon\bigg\}.$$ Then by Theorem \ref{largedev},
    $$\int\psi\, d\mu_{T,k}=\frac{\sum_{\gamma\in\P_T,\,\mu_\gamma\notin \K}\int \psi\,d\mu^k}{\#\P_T}+O(e^{-cT})$$ for some $c>0.$ Since 
    \begin{multline*}
    \frac{\sum_{\gamma\in\P_T,\,\mu_\gamma\notin \K}\int \psi\,d\mu_\gamma^k}{\#\P_T}=(1-O(e^{-cT}))\int\psi\,d\mu^k\\
    +\frac{\sum_{\gamma\in\P_T,\,\mu_\gamma\notin \K}\left(\int \psi\,d\mu_\gamma^k-\int \psi\,d\mu^k\right)}{\#\P_T}, 
    \end{multline*}
    we see that $$\int\psi\,d\mu^k-\epsilon\leq\liminf_{T\to\infty}\int\psi\,d\mu_{T,k}\leq \limsup_{T\to\infty}\int\psi\,d\mu_{T,k}\leq \int\psi\,d\mu^k+\epsilon.$$ Since $\psi$ and $\epsilon$ were arbitrary, the proof is complete.
\end{proof}

The next step is to consider integrability of $f_{D,X}$, which is discussed in Section 4.1 of \cite{KV}.

\begin{lemma}[\cite{KV}]\label{integrability}
    The function $f_{D,X}$ is in $L^1(\Lambda^k,\mu^k).$
\end{lemma}

\begin{remark}
    The result proved in \cite{KV} is much more general than Lemma \ref{integrability}. The flow $X$ need not be Axiom A, and $f_{D,X}$ is integrable with respect to any measure invariant under the product flow. 

\end{remark}
We will complete the proof of Theorem \ref{mainthm1} with an application of the following observation. Let $\Delta_k(\Lambda)$ be the fat diagonal in $\Lambda^k$, i.e. $$ \Delta_k(\Lambda)=\{(x_1,\ldots, x_k)\in \Lambda^k\hbox{ : }x_i=x_j\hbox{ for some }i\neq j\}.$$

\begin{proposition}\label{nested}
    If there exist nested open neighbourhoods $\{B_R\}_{R>0}$ of $\Delta_k(\Lambda)$ such that $\bigcap_{R>0}B_R=\Delta_k(\Lambda)$ and $$\lim_{T\to \infty}\int_{B_R} f_{D,X}\,d\mu_{T,k}=0$$ for sufficiently small $R$, then the following limit exists and the equality holds: $$\lim_{T\to \infty}\int f_{D,X}\,d\mu_{T,k}=\int f_{D,X}\,d\mu^k.$$
\end{proposition}
\begin{proof}
The proof in Lemma 9.7 of \cite{colessharp} can be applied here to show $$\mu^k(\Delta_k(\Lambda))=0.$$ The essential argument is that one can cover the diagonal with products of small balls whose $\mu$ measure decays exponentially fast with the diameter. Using this fact, along with Theorem \ref{measconv} and Lemma \ref{integrability}, one can complete the proof.
\end{proof}
Set $B_R=\bigcup_{x\in \Lambda}\{x\}\times \underbrace{B(x,R)\times \cdots \times B(x,R)}_{k-1\text{ times}}.$ These sets are clearly nested and limit to the diagonal in the required way.

Fix a periodic orbit $\gamma$, and parametrise it as a curve $\gamma(t)=X^t(x_\gamma)$ where $x_\gamma$ is some point on $\gamma$. For every $0\leq t_1\leq \ell(\gamma),$ and $R>0$ define 
$$E_{\gamma}(t_1,R)=\{(t_2,\ldots,t_k)\in [0,\ell(\gamma)]^{k-1}\hbox{ : }\|\gamma(t_i)-\gamma(t_1)\|< R\hbox{ for all }i\}\hbox{, and}$$
$$F_\gamma(t_1,R)=\{(t_2,\ldots,t_k)\in [0,\ell(\gamma)]^{k-1}\hbox{ : }|t_i-t_1|< R\hbox{ for all }i\}.$$
Then we have that $$\int_{B_R} |f_{D,X}|\,d\mu_\gamma^k=\frac{1}{\ell(\gamma)^k}\int_0^{\ell(\gamma)}\left(\int_{E_{\gamma}(t_1,R)}|f_{D,X}(\gamma(t_1),\ldots, \gamma(t_k))|\,dt_2\ldots dt_k\right)dt_1.$$ Since $\Lambda$ is compact and $X$ non-stationary, there is $L>0$ such that $L\leq \|\frac{dX^t}{dt}\|.$ The above then tells us that
\begin{align*}
    \int_{B_R} |f_{D,X}|\,d\mu_\gamma^k &\leq \frac{1}{\ell(\gamma)^k}\int_0^{\ell(\gamma)}\left(\int_{F_\gamma(t_1,\frac{R}{L})}|f_{D,X}(\gamma(t_1),\ldots, \gamma(t_k))|\,dt_2\ldots dt_k\right)dt_1\\
    &\leq \frac{1}{\ell(\gamma)^k}\int_0^{\ell(\gamma)}\int_{-\frac{R}{L}}^{\frac{R}{L}}\cdots \int_{-\frac{R}{L}}^{\frac{R}{L}}|f_{D,X}(\gamma(t_1+s_1),\ldots, \gamma(t_1+s_k))|\,d\textbf{s}\,dt_1.
\end{align*}
The proof of the Key Lemma in \cite{KV} shows that there is some $M>0$ such that for $R$ sufficiently small, and for all $\gamma, t_1,$
$$\int_{-\frac{R}{L}}^{\frac{R}{L}}\cdots \int_{-\frac{R}{L}}^{\frac{R}{L}}f_{D,X}(\gamma(t_1+s_1),\ldots, \gamma(t_1+s_k))\,d\textbf{s}\leq MR.$$
The same argument applies to show that $$\int_{-\frac{R}{L}}^{\frac{R}{L}}\cdots \int_{-\frac{R}{L}}^{\frac{R}{L}}|f_{D,X}(\gamma(t_1+s_1),\ldots, \gamma(t_1+s_k))|\,d\textbf{s}\leq MR.$$By the above, we then have $$\int_{B_R} |f_{D,X}|\,d\mu_\gamma^k\leq \frac{MR}{\ell(\gamma)^{k-1}},$$ and therefore that $$\left|\int_{B_R} f_{D,X}\,d\mu_{T,k}\right|\leq \frac{MR}{(T-1)^{k-1}}.$$ 
Therefore $\lim_{T\to \infty}\int_{B_R} f_{D,X}\,d\mu_{T,k}=0$ and the proof is complete.

\end{document}